\documentclass[12pt,reqno]{amsart}
\usepackage{graphicx}
\usepackage{psfrag}
\usepackage{pxfonts}
\usepackage{mathrsfs}
\usepackage{color}
\usepackage{amsmath,amssymb}

\numberwithin{equation}{section}

\theoremstyle{plain}

\newtheorem{theorem}{Theorem}[section]

\newtheorem{corollary}[theorem]{Corollary}
\newtheorem{proposition}[theorem]{Proposition}
\newtheorem{lemma}[theorem]{Lemma}
\newtheorem{definition}[theorem]{Definition}

\theoremstyle{remark}
\newtheorem{remark}[theorem]{Remark}

\begin{document}

\thanks{}

\author{F. Micena}
\address{Instituto de Matem\'{a}tica e Computa\c{c}\~{a}o,
  IMC-UNIFEI, Itajub\'{a}-MG, Brazil.}
\email{fpmicena82@unifei.edu.br}

\renewcommand{\subjclassname}{\textup{2000} Mathematics Subject Classification}

\date{\today}

\setcounter{tocdepth}{2}

\title{Constant periodic data and rigidity}
\maketitle
\begin{abstract}
In this work we lead with expanding maps of the circle and Anosov diffeomorphisms on $\mathbb{T}^d, d \geq 2.$ We prove that, for these maps, \textit{constant periodic data} imply \textit{same periodic data of these maps and their linearizations}, so in particular we have smooth conjugacy. For expanding maps of the circle and Anosov diffeomorphism on $\mathbb{T}^d, d= 2, 3,$ we have global rigidity. In higher dimensions, $d \geq 4,$ we can establish a result of local rigidity, in several cases. The main tools of this work are celebrated results of rigidity involving same periodic data with linearization and results involving topological entropy of a diffeomorphism  along  an  expanding invariant foliation.
\end{abstract}

\section{Introduction}\label{section.preliminaries}

Expanding endomorphisms of the circle are one of the most studied examples of dynamical systems. A classical class of expanding maps of the circle is the linear model $E_d: S^1 \rightarrow S^1, E_d(x) = dx(mod 1),$ where $d \geq 2$ is a integer number. It is well known that if $f: S^1 \rightarrow S^1$  is an orientation preserving expanding endomorphism of the circle with degree $d \geq 2,$ then $f$ is conjugated to $E_d,$  meaning that there is a homeomorphism $h: S^1 \rightarrow S^1,$ such that $f \circ h = h \circ E_d.$  In particular two orientation preserving expanding endomorphism of the circle $f$ and $g$ with the same degree are conjugated. In the case that $f, g$ above are $C^r, r \geq 2,$ two orientation preserving expanding endomorphism of $S^1,$ it is known by \cite{SS}, that $f$ and $g$ are absolutely continuous conjugated by a conjugacy $h$ if and only if $h$ is $C^r, r \geq 2.$

Other important fact is that every $C^r, r \geq 2$ expanding map $f$ of the circle admit a unique invariant measure $\mu_f$ that is absolute continuous with respect to Lebesgue measure $m$ of  $S^1,$ moreover $\mu_f$ is ergodic. By ergodicity of $\mu_f$ it is possible talk about the Lyapunov exponent with respect to $\mu_f$ for  $f,$ that is $m-$almost everywhere constant and  we will denote by $\lambda_{\mu_f}.$

Now consider $M$ a compact, connected, boundaryless $C^{\infty}$ manifold $M,$ we say that a diffeomorphism $f: M \rightarrow M$ is an Anosov diffeomorphism if $TM$ splits as $TM = E^s_f \oplus E^ u_f$ a continuous and $Df$ invariant spliting, such that $Df$ is uniform contracting on $E^s_f$ and uniform expanding on $E^u_f. $

In this work we lead with  Anosov diffeomorfisms  $f: \mathbb{T}^d \rightarrow \mathbb{T}^d, d \geq 2.$  Denote by $L$ the linearization of $f,$ the map induced on $\mathbb{T}^d$ by the matrix with integer coefficients given by the action of $f$ on $\Pi_1(\mathbb{T}^d).$ It is known by \cite{FRANKS} that $L$ is an Anosov automorphism, and $f$ and $L$ are conjugated by a homeomorphism $h$ such that
 $$h \circ f = L \circ h.$$

Before the results here, we define.

\begin{definition}Let $f: M \rightarrow M$ be a local diffeomorphism. We say that $f$ has constant periodic data if  for any periodic points $p,q$ of $f,$ with period $k$ and $n$ respectively, then $Df^{\tau}(p) = Df^{\tau}(q)$ are conjugated, for every integer $\tau$ such that $f^{\tau}(p) = p$ and $f^{\tau}(q) = q.$ In particular the set of Lyapunov exponents of $p$ and $q,$ are equal and each common Lyapunov exponent has the same multiplicity for both.
\end{definition}

\begin{remark} Constant periodic data is a more weaker condition than to suppose $f$ and its linearization $L$ have same periodic data at corresponding periodic points.
\end{remark}

We are able to prove the following.

\begin{theorem}\label{teo1} Consider $f: S^1 \rightarrow S^1$ a $C^r, r\geq 2$ orientation preserving expanding endomorphism with degree $d \geq 2.$  The map $f$ is $C^r$ conjugated to $E_d$ if and only if  $\lambda_{\mu_f}$ is constant on $Per(f),$ where $Per(f)$ denotes the set of periodic points for $f.$
\end{theorem}

Note that, in the hypothesis of the previous Theorem we don't have suppose $\lambda_{\mu_f}(p) = \log(d).$ In fact Theorem \ref{teo1} generalizes a result by Arteaga in \cite{Art}.

In dimension $d > 1,$ we study regularity of conjugacy of Anosov diffeomorphism. For dimensions two and three we can state the following.

\begin{theorem}\label{teo2}
Consider $f: \mathbb{T}^2 \rightarrow \mathbb{T}^2$ a $C^r, r\geq 2,$ Anosov diffeomorphism. Suppose that for each $\ast \in \{s,u\},$ we have $\lambda^{\ast}_f(p) = \lambda^{\ast}_f(q), $ for any $p,q$ periodic points of $f,$ then $f$ is $C^{r- \varepsilon }$ conjugated with its linearization $L,$ for some $\varepsilon > 0.$
\end{theorem}

\begin{theorem}\label{teo3}
Consider $f: \mathbb{T}^3\rightarrow \mathbb{T}^3$ a $C^r, r\geq 2$ Anosov diffeomorphism. Suppose that $f$ admits a partially hyperbolic structure $ T\mathbb{T}^3 = E_f^s \oplus E_f^{wu} \oplus E_f^{su}$ and  for each $\ast \in \{s,wu, su\},$ we have $\lambda^{\ast}_f(p) = \lambda^{\ast}_f(q), $ for any $p,q$ periodic points of $f,$ then $f$ is $C^{1+\varepsilon} $ conjugated with its linearization $L,$ for some $\varepsilon > 0.$
\end{theorem}

The results above are general, it is sufficient the constant periodic data condition to ensure rigidity. For dimension $d \geq 4,$ we are able to prove a more restricted version of the previous Theorems in several cases. More precisely, we have.

\begin{theorem}\label{teo4} Let $L: \mathbb{T}^d \rightarrow  \mathbb{T}^d , d \geq 4,$ be a linear Anosov automorphism, diaganalizable over $\mathbb{R},$ irreducible over $\mathbb{Q},$ with distinct eigenvalues.  Suppose that $E^s_L = E^s_1 \oplus E^s_2 \oplus \ldots \oplus E^s_k $ and  $E^u_L = E^u_1 \oplus E^u_2 \oplus \ldots \oplus E^u_n .$ If $f$ is a $C^2$ diffeomorphism of $\mathbb{T}^d$ sufficiently $C^1-$close to $L,$ such that $\lambda^u_i(p, f) = \lambda^u_i(q, f), $ for any $p,q \in Per(f),$ $i =1, \ldots, n$ and $\lambda^s_i(p, f) = \lambda^s_i(q, f), $ for any $p,q \in Per(f),$ $i =1, \ldots, k,$ then $f$ is $C^{1+\varepsilon}$ conjugated with its linearization $L,$ for some $\varepsilon > 0.$
  \end{theorem}


\section{One dimensional case - Expanding Maps}

In this section we present some important classical results about expanding endomorphism of the circle, which will be useful for our propose.

\begin{lemma}[Bounded Distortion Lemma] Let $f$ be a $C^{1+ \alpha}$ expanding endomorphism of $S^1.$ There is a constant $C_f \geq 1,$ such that if $I \subset S^1$ is an interval and $f^n$ is injective on $I,$ then
$$\frac{1}{C_f} < \frac{|Df^n(x)|}{|Df^n(y)|} < C_f,$$

for any $x, y \in I.$
\end{lemma}

\begin{proof}
Since $f$ is expanding map, there is $\lambda > 1, $ such that $|Df(z)| > \lambda, $ for every $z \in S^1.$

Consider $x, y \in I$ and denote $x_i = f^i(x), y_i = f^i(y).$ One has

$$\frac{Df^n(x)}{Df^n(y)} = \frac{\prod_{i = 0}^{n-1} Df (f^i (x))}{\prod_{i = 0}^{n-1} Df (f^i (y))} =
\prod_{i = 0}^{n-1} \left( 1 + \frac{Df(x_i) - Df(y_i)}{Df(y_i)}\right).$$

Using mean value theorem and since $f$ is $C^2,$ we obtain

$$\frac{Df^n(x)}{Df^n(y)} = \prod_{i = 0}^{n-1} \left( 1 + \frac{(x_i - y_i)D^2f(z_i)} {Df(y_i)}\right) \leq \prod_{i = 0}^{n-1} (1 + M (\lambda^{-1})^{n-i}|x_n - y_n| ), $$
once $f^{n-i} (x_i) = x_n, f^{n-i} (y_i) = y_n. $ Passing modulus, taking  $\log,$ and using the elementary fact $\log(1 + x) < x,$ for any $x > 0,$ we have

$$ \log\left(\frac{|Df^n(x)|}{|Df^n(y)|}\right ) \leq M |x_n - y_n| \sum_{i = 0}^{\infty} \lambda^{-i} \leq  M \sum_{i = 0}^{\infty} \lambda^{-i} \leq K = K_f, $$
finally

$$\frac{|Df^n(x)|}{|Df^n(y)|} < C_f = \exp(K_f), \frac{|Df^n(y)|}{|Df^n(x)|} < C_f = \exp(K_f). $$

\end{proof}

\begin{theorem}\label{unique}
Any $C^2$ expanding map $f: S^1 \rightarrow S^1$ has a unique absolutely continuous invariant measure, $\mu_f.$ Furthermore $\mu_f$ is ergodic.
\end{theorem}

For a proof, see \cite{K}, page 193.

\begin{theorem}[Sackesteder, Krzyzewsky ] \label{saks}
Let  $f: M \rightarrow M$ be a $C^r, r \geq 2$ expanding map. Then there is a normalized $C^{r-1}$ invariant measure for $f.$
\end{theorem}

\begin{theorem} [Shub-Sullivan's Theorem, \cite{SS}] \label{ss} Let $f,g$ be two $C^r, r \geq 2$ orientation preserving endomorphism of the circle. Then $f$ and $g$ are absolutely continuous conjugated if and only if $f$ and $g$ are $C^r$ conjugated.
\end{theorem}

\begin{remark}The proof that we present here is different from the original one. Here we use an argument involving ordinary differential equations. A similar argument the reader can see in \cite{LLAVE} and \cite{MT2}, for example.
\end{remark}
\begin{proof}
Let $\mu_f$ and $\mu_g$ be the unique invariant measures absolutely continuous with respect to Lebesgue measure on the circle, for $f$ and $g$ respectively. Let $\omega_f$ and $\omega_g$ be the positive densities of $\mu_f$ and $\mu_g$ respectively. Since $h$ is absolutely continuous, the measure $h_{\ast}(\mu_f)$ is a probability invariant measure for $g,$ moreover $h_{\ast}(\mu_f)$ is absolutely continuous. By the uniqueness established in the theorem \ref{unique}, we have $h_{\ast}(\mu_f) = \mu_g.$

Since $h$ is an absolutely continuous homeomorphism, there is $h'$ for almost every point $x \in S^1.$ So using the chang of variable formula, we have:

$$   \int_I \omega_f(x)  dx =  \int_{h(I)} d(h_{\ast}(\mu_f)) = \int_{h(I)} \omega_g dx = \int_I \omega_g(h(x))|h'(x)|dx, $$
for any interval $I \subset S^1.$

So we have $\omega_f(x) = \omega_g(h(x))|h'(x)|,$ for almost everywhere $x \in S^1.$

Since $f,g$ has the same orientation, then $h$ is not reverse wise, then $h' > 0.$ So for almost everywhere $t \in S^1,$ the function $h$ satisfies the ordinary differential equation

\begin{equation}\label{edo}
z' = \frac{\omega_f(t)}{\omega_g(z)}.
\end{equation}

By Theorem $\ref{saks},$ the densities $\omega_f$ and $\omega_g$ are $C^{r-1}$ functions. So considering  equation  $(\ref{edo})$ for every $t \in S^1,$ by classical theory of ordinary differential equations, the O.D.E above has a unique $C^r$ solution $H$ satisfying $H(0) = h(0).$

So, since $h$ has bounded variation, for every $t \in S^1 $ we have

$$h(t) = h(0) + \int_{0}^t h'(s) ds = h(0) + \int_{0}^t \frac{\omega_f(s)}{\omega_g(h(s))}ds,   $$
on the other hand, changing $h$ by $H$ in the equation above, the analogous equations holds for $H.$ Moreover, by Picard's Theorem $H$ is the unique continuous function satisfying the integral equation above with the initial condition $H(0) = h(0).$  Then, by a continuation argument,  $H(t) = h(t),$ for every $t.$

\end{proof}


\subsection{Proof of Theorem \ref{teo1}}Let us to prove Theorem \ref{teo1}.
\begin{proof}  Since the degree of $f$ is  $d \geq 2$ we know that $f$ is conjugated to $E_d(x) = dx ( mod1).$  Fix the notation:  $\log(\lambda), \lambda > 1,$ denoting the Lyapunov exponent on periodic points.  In particular if $p$ has period $n,$ then

$$|Df^n(p)| = \lambda^n. $$
Fix $n > 1$ and $\{{I_{n,j}}\}_{j = 1}^{d^n}$ maximal intervals of injectivity of $f^n,$ such that $f^n(I_{n,j}) = S^1, j=1, 2\ldots, d^n. $
By conjugacy with $E_{d^n},$ in each $\overline{I}_{n,j},$ there is a periodic point $p_{n,j}, f^{n}(p_{n,j}) =  p_{n,j}.$

Denote by $|I|$  the size of the oriented interval $I \subset S^1.$  By bounded distortion lemma we have that there is $C > 1,$ such that

\begin{equation}\label{lips1}
 \frac{1}{C} |Df^{n}(p_{n,j})| \cdot |I_{n,j}|  \leq |S^1| = 1 \leq C |Df^{n}(p_{n,j})| \cdot |I_{n,j}|,
\end{equation}
rewriting we have:

$$ \frac{1}{C} \lambda^n \cdot |I_{n,j}|  \leq 1 \leq C \lambda^n \cdot |I_{n,j}|. $$
In particular, using the inequality above,  we have

$$\frac{1}{C \lambda^n} \leq |I_{n,j}| \leq \frac{C}{\lambda^n},$$
also we have

$$1 \geq d^n \cdot \displaystyle\min_{1 \leq j \leq d^n}\{|I_{n,j}|\} \geq d^n \cdot \frac{1}{C \lambda^n} \Rightarrow \frac{d^n}{\lambda^n} \leq C $$

$$1 \leq d^n \cdot \displaystyle \max_{1 \leq j \leq d^n}\{|I_{n,j}|\} \leq d^n \cdot \frac{C}{\lambda^n} \Rightarrow \frac{d^n}{\lambda^n} \geq \frac{1}{C}, $$
we obtain

\begin{equation}\label{lips2}
\frac{1}{C} \leq \frac{d^n}{\lambda^n} \leq C,
\end{equation}
for every $n \geq 1.$ Since the conjugacy $h$ between $f$ and $E_d$ carries intervals $I_{n,j}$ to intervals $X_{n,j}$ the maximal intervals of injectivity of $E_{d^n},$ which size is $\frac{1}{d^n}.$ So from  equations $(\ref{lips1})$ and $(\ref{lips2})$ have

\begin{equation}\label{lips3}
\displaystyle\frac{1}{C^2} \leq \displaystyle\frac{|h(I_{n,j})|}{|I_{n,j}|} \leq C^2,
\end{equation}
for every $n \geq 1$ and $1 \leq j \leq d^n.$

Since $n$ is arbitrary, by (\ref{lips2}) we have $d^n = \lambda^n = |Df^n(p)|,$ for every $p \in Per(f),$ such that $f^n(p) = p,$ in particular $\lambda_f(p) =\log(d),$ for any $p \in Per(f).$  Also, by $(\ref{lips3}),$ and the fact $\displaystyle\max_{1 \leq j \leq d^n}\{|I_{n,j}|\} \rightarrow 0,$ when $n \rightarrow +\infty, $ we have that $h$  is a bi-Lipschitz conjugacy, consequently an absolutely continuous conjugacy, which it is $C^r$ by Theorem \ref{ss}.

The equality $\lambda_f(p) = \lambda_f(q)$ for any $p, q \in Per(f)$ is obvious when $f$ and $E_d$ are $C^r-$conjugated.

\end{proof}

\begin{corollary} Let $f: S^1 \rightarrow S^1$ be a $C^r, r \geq 2$ orientation preserving, with degree $d \geq 2.$ Denote by $\mathcal{R}_f$ the set of regular points of $f.$ Then $f$ is $C^r$ conjugated to $E_d,$ if and only if $\mathcal{R}_f = S^1.$
\end{corollary}

\begin{proof} Of course if $f$ and $E_d$ are $C^r$ conjugated then $\mathcal{R}_f = S^1.$ On the other hand, suppose that $\mathcal{R}_f = S^1$ and the Lyapunov exponent is not constant on $Per(f).$ Let $\lambda(p) \neq \lambda(q)$ be two different Lyapunov exponents at periodic points $p$ and $q.$ Then using 'symbolic representation' it is possible to obtain $x \in S^1$ and two subsequences $\{n_k\}_k$ and $\{m_k\}$ such that

$$\lim_{k \rightarrow +\infty} \frac{1}{n_k} \log ||Df^{n_k}(x) || = \lambda(p),$$

$$\lim_{k \rightarrow +\infty} \frac{1}{m_k} \log ||Df^{m_k}(x) || = \lambda(q),$$
so $x \notin \mathcal{R}_f.$
 \end{proof}

\section{General Preliminaries on Anosov Diffeomorphism}

First let us define basic concepts. Let $M$ be a $C^{\infty}$ riemannian closed (compact, connected and boundaryless)  manifold.
A $C^1-$diffeomorphism $f: M \rightarrow M$ is called an absolute partially hyperbolic diffeomorphism if the tangent bundle $TM$ admits a $Df$ invariant tangent decomposition $TM =  E^s \oplus E^c \oplus E^u$ such that all unitary vectors $v^{s} \in E^{s}_x, v^{c} \in E^{c}_y, v^{u} \in E^{u}_z,$ for every $x,y,z \in M$ satisfy:

$$ ||D_x f v^s || < ||D_y f v^c || < ||D_z f v^u ||,$$
moreover,

$$||D_x f v^s || < 1 \;\mbox{and}\; ||D_z f v^u || > 1, $$
for a suitable norm.

When $TM =  E^s  \oplus E^u,$ where $E^s$ and $E^u$ is as above, then $f$ is called an Anosov diffeomorphism.

\begin{definition}[SRB measure] Let $f: M \rightarrow M$ be a $C^2$ diffeomorphism. An $f$ invariant Borel probability measure $\mu$ is called an SRB measure if $(f, \mu)$ has a positive Lyapunov exponent a.e. and $\mu$ has absolutely continuous conditional measures on unstable manifolds.
\end{definition}

\begin{theorem}[Ledrappier-Young, \cite{LY}] Let $f$ be a $C^2-$diffeomorphism and $\mu$ an $f-$invariant probability measure with a positive Lyapunov exponent a.e. Then $\mu$ is SRB measure of $f$ if and only if
$$h_{\mu}(f) = \int \sum_{\lambda_i > 0} \lambda_i \dim(E_i) d\mu. $$
\end{theorem}

\begin{theorem}[Bowen, \cite{Bo75}] Let $f: M \rightarrow M$ be a $C^2-$Anosov diffeomorphism. Suppose that for every periodic point $p,$ with period $n_p$ we have $|\det(Df^{n_p}(p))| = 1,$ then $f$ preserves a unique borelian probability measure $\mu$ absolutely continuous with respect to Lebesgue measure $m,$ moreover $\frac{d\mu}{dm} = H $ is a positive H\"{o}lder function.
\end{theorem}

It is known by \cite{Bo75} that a $C^2-$axiom A diffeomorphism has a unique SRB measure. It includes transitive Anosov diffeomorphisms, particularly Anosov diffeomorphism on $\mathbb{T}^d.$

\begin{lemma}[Anosov Closing Lemma] \label{closing}
Let $f: M \rightarrow M$ be a $C^{1+\alpha}$ diffeomorphism preserving a hyperbolic Borel probability measure. For all $\delta > 0 $ and $\epsilon > 0$ there exists $\beta= \beta(\delta, \epsilon) >0$ such that if $x, f^{n(x)}(x) \in \Delta_{\delta}$ (Pesin block) for some $n(x) >0$ and $d(x, f^{n(x)}(x)) < \beta$ then there exists a hyperbolic periodic point of period $n(x)$, $z$ with $d(f^k(x), f^k(z)) \leq \epsilon$ for all $0 \leq k \leq n(x)-1.$
 \end{lemma}

Note that, for the Anosov diffeomorphism $f,$ we may take (a.e)  $\Lambda_{\delta} = \mathbb{T}^d$ for some $\delta>0.$

\begin{definition}[Specification Property] Let $f: M \rightarrow M$ be a diffeomorphism. We say that $f$ has the specification property if given $\varepsilon > 0$ there is a relaxation time $N\in \mathbb{N}$  such that every $N-$spaced collection of orbit segments is $\varepsilon-$shadowed by an actual orbit. More precisely, for points $x_1, x_2, \ldots, x_n$ and legths $k_1, \ldots, k_n \in \mathbb{N}$ one can find times $a_1, \ldots, a_n$ such that $a_{i+1} \leq a_i + N$ and a point $x$ such that $d(f^{a_i + j}(x), f^{j}(x_i) ) < \varepsilon $ whenever $0 \leq j \leq k_i.$ Moreover, one can choose $x$ a periodic point with period no more than $a_n + k_n + N.$
\end{definition}

\begin{theorem}[Bowen, \cite{Bo74}] Every transitive Anosov diffeomorphism has the specification property.
\end{theorem}

Let us speak about rigidity of Anosov diffeomorphisms of the torus, that we will use in this work.

\begin{theorem}[De La Llave, \cite{LLAVE}] Let $f,g$ be two $C^k, k\geq2,$ Anosov diffeomorphism of $\mathbb{T}^2$ and $h$ a homeomorphism of $\mathbb{T}^2,$ satisfying
$$h\circ f = g\circ h.$$
If the Lyapunov exponents at corresponding periodic orbits are the same, then $h \in C^{k - \varepsilon}.$
\end{theorem}

Rigidity in three dimensional torus was studied in \cite{GoGu}. We use indirectly the following result, that is strongly used in \cite{MT2}.

\begin{theorem}[Gogolev-Guysinsky, \cite{GoGu}] Let f and g be Anosov diffeomorphisms of $\mathbb{T}^3$ and $$h\circ f = g\circ h,$$
where $h$ is a homeomorphism homotopic to identity. Suppose that periodic data of $f$ and $g$ coincide, meaning  Lyapunov exponents at corresponding periodic orbits are the same.
Also assume that $f$ and $g$ can be viewed as partially hyperbolic diffeomorphisms:
$$E^s_g \oplus E^{wu}_g \oplus E^{su}_g = T\mathbb{T}^3 = E^s_f \oplus E^{wu}_f \oplus E^{su}_f.$$
Then the conjugacy $h$ is $C^{1 + \nu},$ for some $\nu > 0.$
\end{theorem}

In higher dimensional torus $\mathbb{T}^d, d \geq 4,$ same periodic data does not imply regularity of the conjugacy, De La Llave in \cite{LLAVE}  constructed two Anosov diffeomorphisms on $\mathbb{T}^4, $ with the same periodic data which are only
H\"{o}lder conjugated. Saghin and Yang in \cite{SY} proved.

\begin{theorem}Let $L$ be an irreducible Anosov automorphism of $\mathbb{T}^d, d\geq 3, $ with simple real
spectrum. If $f$ is a $C^2$
volume preserving diffeomorphism $C^1$ close to $L$ and has
the same Lyapunov exponents of $L,$ at corresponding periodic points,  then $f$ is $C^{1+\varepsilon}$
conjugated to $L$ for some $ \varepsilon > 0.$
\end{theorem}

\section{Two dimensional case}

 First we remember the statement of Theorem \ref{teo2}.
\begin{theorem}
Consider $f: \mathbb{T}^2 \rightarrow \mathbb{T}^2$ a $C^r, r\geq 2,$ Anosov diffeomorphism. Suppose that for each $\ast \in \{s,u\},$ we have $\lambda^{\ast}_f(p) = \lambda^{\ast}_f(q), $ for any $p,q$ periodic points of $f,$ then $f$ is $C^1 $ conjugated with its linearization $L.$
\end{theorem}

\begin{proof} Given $\mu$ an $f-$invariant probability measure, by the Closing lemma for hyperbolic measures, we have $\lambda^{\ast}_{\mu}(x, f) =  \lambda^{\ast}_f(p), \ast \in \{s,u\},$ for $\mu$ almost everywhere $x \in \mathbb{T}^2,$  here $p$ is a given periodic point for $f. $
By Ruelle formula we have $$h_{\mu}(f) \leq \lambda^u_f(p).$$
Since $f$ is $C^r, r \geq 2, $ there is a unique SRB probability measure $\nu$ invariant for $f.$ In this case
$$h_{\nu}(f) = \int_{\mathbb{T}^2} \lambda^{u}_{\mu}(x, f) = \lambda^u_f(p). $$

So $\nu $ is the maximal entropy measure, then by variational principal we have $$\lambda^u_{L} = h_{top}(L)= h_{top}(f) =  h_{\nu}(f) = \lambda^{u}_f(p).$$
Taking $f^{-1},$ analogously we conclude that $$\lambda^{s}_f(p) = \lambda^{s}_{L} ,$$ for any $p \in Per(f).$

By the De La Llave result in \cite{LLAVE},  we have $f$ and $L$ are $C^{r -\varepsilon}, $ conjugated for some $\varepsilon > 0. $
\end{proof}

\section{Three dimensional case}

Let us first state the following lemma.

\begin{lemma}\label{constant}
Let $f: M \rightarrow M  $ be a $C^2-$Anosov diffeomorphism with constant periodic data. Then $f$ is conservative and the absolutely continuous measure is the maximal entropy measure.
\end{lemma}

\begin{proof}
Let $x \in M$ be a regular point, define $\Lambda^s(x)$ be the sum of negative Lyapunov exponents of the point $x$ and  $\Lambda^u(x),$  the sum of positive Lyapunov exponents of the point $x.$
Since we have constant periodic data, for any $p,q \in Per(f),$ holds $\Lambda^{\ast}(p) = \Lambda^{\ast}(q), \ast \in \{s,u\} .$ We call $\Lambda^{\ast} $ the value $\Lambda^{\ast}(p), \ast \in \{s,u\},$ where $p$ is a periodic point of $f.$

Let $\mu $ be an $f$-invariant probability measure. Obviously $\mu$ is a hyperbolic measure for $f.$ So, using Anosov Closing lemma, for $\mu$ a.e. $x \in M$ we have
$$ \Lambda^{\ast}(x) =  \Lambda^{\ast}, \ast \in \{s,u\}. $$

By Ruelle formula, we have:

$$h_{\mu}(f) \leq \Lambda^u ,$$
for any $\mu$ an $f-$invariant probability measure.

Let $\mu_f$ be a SRB measure of $f.$ So, for $\mu_f$ we also have $\Lambda^{\ast}(x) = \Lambda^{\ast},$ for $\mu_f$ a.e. $x \in M.$ Using the SRB property of $\mu_f,$ we get

$$h_{\mu_f}(f) = \int_M \Lambda^u(x)d\mu_f = \Lambda^u,$$
by variational principle $h_{\mu_f}(f) = h_{top}(f) = \Lambda^u.$ So $\mu_f$ is the maximal entropy measure of $f.$ Analogously, taking $f^{-1} ,$ we conclude that $\mu_{f^{-1}}$ is also maximal entropy measure of $f,$ so $\mu_f = \mu_{f^{-1}}$ and
$$ -\Lambda^s= h_{\mu_{f^{-1}}}=h_{\mu_f}(f) = h_{top}(f) = \Lambda^u,$$
then
$$\Lambda^s + \Lambda^u = 0,$$
by Bowen result in \cite{Bo75}, we conclude that $f$ is conservative and $\mu_f$ is the absolutely continuous measure for $f.$
\end{proof}

The conservative case of Theorem \ref{teo3} is proved in \cite{MT2}.

\begin{lemma}[Corollary 3.3 of \cite{MT2}]\label{micena}
  Let $f: \mathbb{T}^3 \rightarrow \mathbb{T}^3 $ be a $C^r, r \geq 2,$ volume preserving Anosov diffeomorphism such that $T\mathbb{T}^3 =  E^{s}_f \oplus E^{wu}_f \oplus E^{su}_f.$ Suppose that there are constants $\Lambda^{\sigma}_f,  \sigma \in \{s, wu, su\},$ such that  for any $p \in Per(f)$ we have $\lambda^{\sigma}_f(p) = \Lambda^{\sigma}_f, \sigma \in \{s, wu, su\}.$ Then $f$ is $C^1$ conjugated with its linearization $L.$
\end{lemma}

It is important to note that Lemma \ref{micena} holds in conservative context and the proof is the same. In fact the conjugacy above is $C^{1+ \varepsilon},$ for some $\varepsilon > 0,$ because \cite{GoGu}.

Then  Theorem \ref{teo3}, follows from the above previous lemmas. In fact, since in Theorem \ref{teo3} we are supposing constant periodic data, by Lemma \ref{constant} we have $f$ is conservative. So we can apply Lemma \ref{micena} to conclude that the conjugacy is $C^{1+ \varepsilon},$ for some $\varepsilon > 0.$

\begin{corollary} Let $f$ be as in Theorem \ref{teo2}( or Theorem \ref{teo3}). If every $x$ is regular, then  $f$ is $C^{1+\varepsilon} $ conjugated with its linearization $L,$ for some $\varepsilon > 0.$
\end{corollary}

\begin{proof} It is as consequence of Specification Property of Anosov diffeomorphisms. In fact, if every point is regular, then the periodic data of $f$ is constant.
Suppose that $p, q \in Per(f)$ and for some $\ast$ we have $\lambda_f^{\ast}(p) < \lambda_f^{\ast}(q), $ using specification it is possible to find a  point $z$ and subsequences $\{n_k\}_{k=1}^{+\infty}$ and $\{m_k\}_{k=1}^{+\infty},$ such that for a given small $\varepsilon$ we have

$$\frac{1}{n_k} \log(Df^{n_k}(z)| E^{\ast}_f ) < \lambda_f^{\ast}(p) + \varepsilon < \lambda_f^{\ast}(q) - \varepsilon <
 \frac{1}{m_k} \log(Df^{m_k}(z)| E^{\ast}_f ), $$
for all $k \geq 1.$ So $z$ can not be a regular point.

\end{proof}

\section{Higher dimensions}

In this section we will use results about entropy a diffeomorphism $f$ along an expanding and $f-$invariant foliation from \cite{Hua} and the tools presented in \cite{Ca} and citation therin.

\begin{proposition}\label{propuniform} Let $L: \mathbb{T}^d \rightarrow \mathbb{T}^d$ be an Anosov linear automorphism, diagonalizable and irreducible over $\mathbb{Q}.$ If $f$ is a $C^2-$Anosov diffeomorphism sufficiently $C^1-$close to $L$ with constant periodic data, then every point $x \in \mathbb{T}^d$ is regular and $\lambda^{\ast}_f(x) = \lambda^{\ast}_f(p),$ where $p $ is a given periodic point of $f.$ Moreover the convergence as in the definition of Lyapunov exponent is uniform, for all possible invariant direction.
\end{proposition}

We prove the proposition after some lemmas, in the same lines as in \cite{Ca}.

\begin{lemma} Let $\mathcal{M}$ be the space of $f-$invariant measures, $\phi$ be a continuous function on $M.$ If $\int \phi d\mu < \lambda, \; \forall \mu \in \mathcal{M}, $ then for every $x \in M,$ there exists $n(x)$ such that
$$\frac{1}{n(x)} \sum_{i = 0}^{n(x) - 1} \phi(f^i(x)) < \lambda. $$
\end{lemma}

\begin{proof}
See \cite{Ca}.
\end{proof}

\begin{lemma}\label{lemmauniform} Let $\mathcal{M}$ be the space of $f-$invariant measures, $\phi$ be a continuous function on $M.$ If $\int \phi d\mu < \lambda, \; \forall \mu \in \mathcal{M}, $ then  there exists $N$ such that for all $n \geq N,$ we have
$$\frac{1}{n} \sum_{i = 0}^{n - 1} \phi(f^i(x)) < \lambda,$$
 for all $x \in M.$
\end{lemma}

\begin{proof}
See \cite{Ca}.
\end{proof}
In the previous lemmas if we replace  $\int \phi d\mu < \lambda$ by $\int \phi d\mu > \lambda,$ we can get analogous statements.
Let us to prove  Proposition \ref{propuniform}

\begin{proof}[of the Proposition \ref{propuniform}]
We go to use $\phi = \log |Df|E^{\ast}_f (\cdot)|.$ Since we have ergodic decomposition, we also go to consider $\mu$ an ergodic and $f-$invariant measure. The diffeomorphism $f$ is Anosov, then every $\nu \in \mathcal{M}$ is a hyperbolic measure.   Fix $p$ a periodic point of $f.$ Since $f$ has constant periodic data, by Anosov Closing Lemma, we have $\lambda^{\ast}_f(x) = \lambda^{\ast}_f(p), $ for $\mu$ a.e. $x \in  \mathbb{T}^d.$  Now, by Ergodic Birkhoff Theorem, we have

$$\lambda^{\ast}_f(p) = \int \log |Df|E^{\ast}_f (x)|d\mu(x), $$
for any $\mu \in \mathcal{M},$  ergodic measure.
Consider $\varepsilon > 0,$  we can apply the previous lemmas with number $\lambda^{\ast}_f(p) + \varepsilon. $ So we obtain that there is a integer $N_1 > 0,$ such that
$$\frac{1}{n}\sum_{i=1}^{n-1} \phi(f^i(x)) < \lambda^{\ast}_f(p) + \varepsilon, $$ for any $x \in \mathbb{T}^d$ and $n \geq N_1.$
So we have
\begin{equation}\label{menor}
\frac{1}{n} \log(Df|E^{\ast}_f(x)) < \lambda^{\ast}_f(p) + \varepsilon, \, \forall x \in \mathbb{T}^d \; \mbox{and}\;  n \geq N_1.
\end{equation}

 Using the previous lemmas in their versions $`<`,$ we can find $N_2 > 0,$ integer such that
\begin{equation}\label{maior}
\frac{1}{n} \log(Df|E^{\ast}_f(x)) > \lambda^{\ast}_f(p) - \varepsilon, \, \forall x \in \mathbb{T}^d \; \mbox{and}\;  n \geq N_2.
\end{equation}

Taking $N = N_1 + N_2$ and using the equations $(\ref{menor})$ and $(\ref{maior}),$ we conclude that $\frac{1}{n} \log(Df|E^{\ast}_f(x))$  converges uniformly to $\lambda^{\ast}_f(p).$

\end{proof}

In Theorem \ref{teo4} we can suppose that the eigenvalues of $L$ satisfying $0< |\beta_1^s| < \ldots < |\beta_k^s| < 1 <  |\beta_1^u| < \ldots < |\beta_n^u|. $ The Lyapunov exponents of $L,$ are
$\lambda^s_i(L) = \log(|\beta_i^s|), i =1, \ldots, k$ and $\lambda^u_i(L) = \log(|\beta_i^u|), i =1, \ldots, n.$ For $f$ we denote by $\lambda^u_{i}(x,f)$ the Lyapunov exponent of $f$ at $x$ in the direction $E^{u,f}_i, i = 1, \ldots, n$ and by  $\lambda^s_{i}(x,f)$ the Lyapunov exponent of $f$ at $x$ in the direction $E^{s,f}_i, i = 1, \ldots, k,$ in the cases that Lyapunov exponents are defined.

Let us introduce a notation $E^{s,L}_{(1, i)} = E^s_1 \oplus \ldots \oplus E^s_i, i=1, \ldots, k$ and  $E^{u,L}_{(1, i)} = E^u_1 \oplus \ldots \oplus E^u_i, i=1, \ldots, n.$ If $j > i,$ we denote $E^{s,L}_{(i, j)} =  E^s_i \oplus \ldots \oplus E^s_j $ and $E^{u,L}_{(i, j)} =  E^u_i \oplus \ldots \oplus E^u_j. $

In the setting of Theorem \ref{teo4}, it is known by Pesin \cite{pesin2004lectures}, that if $f$ is $C^1-$close to $L,$ then $T\mathbb{T}^d$ admits a similar splitting
$E^s_f = E^{s,f}_1 \oplus E^{s,f}_2 \oplus \ldots \oplus E^{s,f}_k $ and  $E^u_L = E^{u,f}_1 \oplus E^{u,f}_2 \oplus \ldots \oplus E^{u,f}_n .$ As before, define $E^{u,f}_{(1,i)} = E^{u,f}_1 \oplus  \ldots \oplus E^{u,f}_i $ and $E^{s,f}_{(1,i)} = E^{s,f}_1 \oplus  \ldots \oplus E^{s,f}_i ,$ analogously, for $i > j,$ we define $E^{s,f}_{(i, j)}$  and $E^{u,f}_{(i, j)}.$

By continuity of each subbundle,  we can take the decomposition $E^s_f \oplus E^{u,f}_{(1,i)} \oplus E^{u,f}_{(i+1, n)}$ such that it is a uniform partially hyperbolic splitting.

Moreover, by \cite{B}, each $E^{u,f}_{(1,i)} = E^{u,f}_1 \oplus  \ldots \oplus E^{u,f}_i, $ is integrable to an invariant foliation $W^{u,f}_{(1, i)},$ with $i =1, \ldots, n.$ An analogous construction holds for stable directions. By \cite{FPS}, since $f$ is $C^1-$close to $L,$ the conjugacy $h$ between $L$ and $f$ is such that $h(W^{u,L}_{(1,i)}) = W^{u,f}_{(1,i)}, $ the same for stable foliations.

In \cite{Hua} the authors lead with a notion of topological entropy $h_{top}(f, \mathcal{W} )$ of an invariant expanding foliation $\mathcal{W}$ of a diffeomorphism $f. $ They establish variational principle in this sense and relation between $h_{top}(f, \mathcal{W} )$ and volume growth of $\mathcal{W}. $

Here $W(x)$ will denote the leaf of $\mathcal{W}$ by $x.$ Given  $\delta  > 0,$  we denote by $W(x, \delta)$ the $\delta-$ball centered in $x$ on $W(x),$ with the induced riemannian distance, that we will denote by $d_{W}.$

Given $x \in M, $ $\varepsilon > 0, $ $\delta > 0$ and $n \geq 1$ a integer number, let $N_{W}(f, \varepsilon, n, x, \delta)$ be the maximal number of points in $\overline{W(x, \delta)}$ such that $\displaystyle\max_{j =0 \ldots, n-1} d_{W}(f^j(x), f^j(y)) \leq \varepsilon.$

\begin{definition} The unstable entropy of $f$ on $M,$ with respect to the expanding foliation $\mathcal{W}$ is given by
$$h_{top}(f, \mathcal{W} ) = \lim_{\delta \rightarrow 0} \sup_{x \in M} h^{\mathcal{W}}_{top}(f, \overline{W(x, \delta)}), $$
where
$$h^{\mathcal{W}}_{top}(f, \overline{W(x, \delta)}) = \lim_{\varepsilon \rightarrow 0} \limsup_{n \rightarrow +\infty} \frac{1}{n} \log(N_{W}(f, \varepsilon, n, x, \delta)). $$
\end{definition}

Define $\mathcal{W}-$volume growth by
 $$\chi_{\mathcal{W}}(f) = \sup_{x \in M } \chi_{\mathcal{W}}(x, \delta), $$
where
$$ \chi_{\mathcal{W}}(x, \delta) = \limsup_{n\rightarrow +\infty} \frac{1}{n} \log(Vol(f^n(W(x, \delta)))).$$

Note that, since we are supposing $\mathcal{W}$ a expanding foliation, the above definition is independent of $\delta$ and the riemannian metric.

\begin{theorem}[Theorem C of \cite{Hua}]\label{teoH} With the notations above we have
$$h_{top}(f, \mathcal{W} ) = \chi_{\mathcal{W}}(f).$$
\end{theorem}

As a consequence of Proposition \ref{propuniform} and  Theorem \ref{teoH}, we have.

\begin{corollary} If $f$ as in Theorem \ref{teo4} and $W^{u,f}_{(1, i)}$ is the foliation tangent to $E^{u,f}_{(1, i)},$ then $h_{top}(f, W^u_{(1, i)}) = \displaystyle\sum_{j=1}^i \lambda^{u}_j(p, f).$
\end{corollary}

\begin{proof} Fix $p$ a periodic point of $f.$ By Proposition \ref{propuniform}we have
$$\lim_{n \rightarrow + \infty}\frac{1}{n} \log(Vol(f^n((W^{u,f}_{(1, i)}(x, \delta)))) = \lim_{n \rightarrow +\infty}\frac{1}{n} \log(|\det(Df^n(x)|E^{u,f}_{(1,i)}(x))\cdot Vol(W^{u,f}_{(1, i)}(x, \delta)) |).$$
Using Proposition \ref{propuniform}, the right side of the above expression converges uniformly to $\sum_{j=1}^i \lambda^u_{j}(p,f).$ So, by Theorem C of \cite{Hua}, we have  $h_{top}(f, W^u_{(1, i)}) = \displaystyle\sum_{j=1}^i \lambda^{u}_j(p,f),$ as required.
\end{proof}

We are ready to prove Theorem \ref{teo4}.

\begin{proof}  Since $h(W^{u,L}_{(1, i)}) =  W^{u,f}_{(1, i)},$ we have $h_{top}(f, W^{u,f}_{(1, i)}) = h_{top}(L, W^{u,L}_{(1, i)})  .$ Now, consider $ \beta^s_i, i =1, \ldots, k,$ the eigenvalues of $L,$ we have $$0 < |\beta^s_1| <  |\beta^s_2| < \ldots < |\beta^s_k|< 1$$ and  $ \beta^u_i, i =1, \ldots, n,$ such that $$1 < |\beta^u_1| <  |\beta^u_2| < \ldots < |\beta^s_n| .$$
Let $p$ be a periodic point of $f.$  Since we have constant periodic data, so for any $i =1, \ldots, n$ we have
$$\lambda^u_1(p, f) + \ldots + \lambda^u_i(p, f) =  h_{top}(f, W^u_{(1, i)}) = h_{top}(L, W^u_{(1, i)}(L) ) = \lambda^u_1(L) + \ldots + \lambda^u_i(L), $$
 for any $i =1, \ldots, n.$
So, for $i = 1,$ we have
$$\lambda^u_1(p, f) = \lambda^u_1(L), $$
for $i = 2,$ we have $\lambda^u_1(p, f) + \lambda^u_2(p, f)  =  \lambda^u_1(L) + \lambda^u_2(L),$ since $\lambda^u_1(p, f) = \lambda^u_1(L),$ we get
$$\lambda^u_2(p, f) = \lambda^u_2(L).$$
Analogously $\lambda^u_i(p, f) = \lambda^u_i(L), i =1, \ldots, n.$

Taking the  inverses, we obtain $$\lambda^s_i(p, f) = \lambda^s_i(L), i =1, \ldots, k,$$
so, $f$ and $L$ has the same periodic data, by \cite{Go} and \cite{SY}, the maps $f$ and $L$ are $C^{1+ \varepsilon}$ conjugated for some $\varepsilon > 0,$  if $f$ is enough  $C^1-$close to $L.$

\end{proof}

As in Corollary 5.3 we have.

\begin{corollary} Let $L : \mathbb{T}^d \rightarrow \mathbb{T}^d  $ be as in Theorem \ref{teo4}. If $f$ is a $C^r, r \geq 2,$ diffeomorphism, $C^1-$close to $L $ and every $x \in \mathbb{T}^d  $ is regular point of $f,$ then $f$ and $L$ are $C^{1+\varepsilon}-$conjugated, for some $\varepsilon > 0. $
\end{corollary}

The proof is similar to Corollary 5.3.

\end{document}